\documentclass[a4paper,12pt]{article}
\usepackage[T2A]{fontenc}                      
\usepackage[cp1251]{inputenc}           
\usepackage[all]{xy}
\usepackage{amssymb}
\usepackage{cite}
\usepackage{cmap}
\usepackage{latexsym}
\usepackage{enumerate}
\usepackage{amsmath, amsthm, amscd, amsfonts, amssymb, graphicx, color}
\usepackage[left=2.5cm,right=2.5cm,top=2cm,bottom=2cm]{geometry}
\usepackage{indentfirst}
\usepackage{array}
\usepackage{bm}
\usepackage{float}

\usepackage{wrapfig}

\usepackage{authblk}
\usepackage{enumitem}

\newtheorem{theorem}{Theorem}[section]
\newtheorem{lemma}{Lemma}[section]

\usepackage[colorlinks=true, linkcolor=blue, citecolor=blue, urlcolor=blue]{hyperref}

\title{On the existence of $F$-irregular graphs for disconnected graphs $F$}

\author{Tatiana~Dovzhenok}

\affil{\small Research Laboratory ``Mathematics of Hybrid Intelligence Systems'', \\
	Francisk Skorina Gomel State University, Gomel, 246028, Belarus}

\date{} 

\begin{document}	
\maketitle

\begin{abstract}
	For a given graph $F$, a graph $G$ is $F$-irregular if all its vertices have pairwise distinct $F$-degrees, where the $F$-degree of a vertex $v$ is the number of subgraphs of $G$ containing $v$ and isomorphic to $F$. While the existence of $F$-irregular graphs has been actively studied for connected graphs $F$, the case where $F$ is disconnected remains entirely unexplored. In this paper, we provide a complete solution to this problem. We prove that for every disconnected graph $F$ with $|E(F)| \ge 2$, there exist infinitely many $F$-irregular graphs. Conversely, if $|E(F)| \le 1$, we show that no non-trivial $F$-irregular graph exists.
\smallskip
	
\textbf{Keywords:} $F$-irregular graph, $F$-degree of a vertex, disconnected graph, graph irregularity, Strong conjecture about $F$-irregular graphs.
\end{abstract}
	
\section{Introduction}
All graphs considered in this paper are simple, finite, and undirected. For a fixed graph $F$, the $F$-degree of a vertex $v$ in a graph $G$ is the number of subgraphs of $G$ isomorphic to $F$ that contain $v$. A graph $G$ is $F$-irregular if all of its vertices have pairwise distinct $F$-degrees. 

The study of $F$-irregular graphs dates back to the work of Chartrand et al.~\cite{1} (see also the monograph~\cite{2}), who conjectured that a non-trivial $F$-irregular graph exists for every connected graph $F$ on at least three vertices. The development of this problem led Dovzhenok et al.~\cite{3} to introduce the Strong conjecture, which asserts that the number of such graphs is infinite. This conjecture was verified for several infinite classes of graphs, including 2-connected graphs $F$ with $\delta(F)=2$~\cite{3} and graphs $F$ of diameter 2~\cite{4}.
\smallskip

In September 2026, Schreib~\cite{5} presented a proof of the existence conjecture, while Dovzhenok and Filuta~\cite{6} independently proposed a proof of the Strong conjecture. In light of~\cite{7}, however, the case of complete graphs~$F$ in~\cite{5} still needs to be explicitly addressed.
\smallskip

In contrast to the connected framework, this paper focuses on the existence and multiplicity of $F$-irregular graphs for \emph{disconnected} graphs~$F$. Our main result provides a complete solution to this problem, as formalized below.

\begin{theorem}\label{t1.1}
	Let $F$ be a disconnected graph.
	\begin{enumerate}
		\item If $|E(F)| \ge 2$, then there exist infinitely many $F$-irregular graphs.
		\item If $|E(F)| \le 1$, then no non-trivial $F$-irregular graph exists.
	\end{enumerate}
\end{theorem}

To prove Theorem~\ref{t1.1} for $|E(F)| \ge 2$, we employ the explicit constructions and methods developed in~\cite{6}.
\smallskip

We adopt standard notation: $V(G)$ and $E(G)$ denote the vertex and edge sets of a graph $G$, respectively; $|G|$ is its order, while $\delta(G)$ and $\Delta(G)$ represent its minimum and maximum degrees. The complete graph on $m$ vertices is given by $K_m$, and its complement, the edgeless graph, by $\overline{K_m}$. A connected component of $G$ isomorphic to $K_m$ is called a \emph{$K_m$-component}. Following the convention in~\cite{6}, an edge between vertices $u$ and $v$ is denoted as $(u,v)$. We let $\mathbb{N}$ and $\mathbb{Z}^+$ denote the sets of positive and nonnegative integers, respectively.
\smallskip

Section~2 addresses the effect of isolated vertices in $F$ on $F$-irregularity and provides the proof of Theorem~\ref{t1.1} for $|E(F)| \le 1$. Sections~3 and~4 examine the existence case for disconnected $F$ under the conditions $\delta(F) = 1$ and $\delta(F) \ge 2$, respectively. Finally, Section~5 synthesizes these results to complete the proof of Theorem~\ref{t1.1}.

\section{The effect of isolated vertices in $F$ on $F$-irregularity}

We first consider the case where $F = \overline{K_m}$, and then analyze how the presence of isolated vertices in $F$ affects the problem when $F$ contains edges.

\begin{lemma}\label{l2.1}
	For any integer $m \ge 1$, there exists no non-trivial $\overline{K_m}$-irregular graph.
\end{lemma}

\begin{proof}
	Let $G$ be an arbitrary non-trivial graph of order $g = |G|$. Since $\overline{K_m}$ is an edgeless graph, the $\overline{K_m}$-degree of any vertex in $G$ is exactly $\binom{g-1}{m-1}$. Thus, all vertices in $G$ have identical $\overline{K_m}$-degrees, meaning that $G$ cannot be $\overline{K_m}$-irregular.
\end{proof}

\begin{theorem}\label{t2.1}
	Let $F = H + \overline{K_m}$ for some integer $m \ge 1$, where $H$ is a nonempty graph with no isolated vertices. A non-trivial graph $G$ is $F$-irregular if and only if $|G| > |F|$ and $G$ is $H$-irregular.
\end{theorem}

\begin{proof}
	Let $g = |G|$ and $h = |H|$. If $2 \le g \le |F|$, the $F$-degree of every vertex in $G$ is trivially zero, so $G$ cannot be $F$-irregular. Since the condition $|G| > |F|$ is necessary, we assume for the remainder of the proof that $g \ge |F| + 1 = h + m + 1$.
	
	For any vertex $v \in V(G)$, let $x$ and $y$ denote its $F$-degree and $H$-degree in $G$, respectively, and let $k$ be the total number of copies of $H$ in $G$. Each copy of $F$ containing $v$ either includes $v$ as an isolated vertex in the copy of $\overline{K_m}$ or as a vertex in the copy of $H$. The former case requires a copy of $H$ that avoids $v$, supplemented by $m-1$ isolated vertices chosen from the remaining $g-h-1$ vertices of $G$, which yields $(k-y)\binom{g-h-1}{m-1}$ configurations. The latter case requires a copy of $H$ containing $v$, supplemented by $m$ isolated vertices chosen from the remaining $g-h$ vertices of $G$, yielding $y\binom{g-h}{m}$ configurations. It follows that
	\[
	\begin{aligned}
		x &= (k-y) \binom{g-h-1}{m-1} + y \binom{g-h}{m} \\
		&= y \binom{g-h-1}{m} + k \binom{g-h-1}{m-1},
	\end{aligned}
	\]
	where the coefficient $\binom{g-h-1}{m}$ is strictly positive since $g \ge h + m + 1$. Consequently, all vertices in $G$ have pairwise distinct $F$-degrees if and only if they have pairwise distinct $H$-degrees, which completes the proof.
\end{proof}

\begin{lemma}\label{l2.2}
	If $F = K_2 + \overline{K_m}$ for some integer $m \ge 1$, then there exists no non-trivial $F$-irregular graph.
\end{lemma}

\begin{proof}
	The statement follows from Theorem~\ref{t2.1} and the fact that no non-trivial graph has pairwise distinct vertex degrees, so none is $K_2$-irregular.
\end{proof}

\section{The case of disconnected graphs $F$ with $\delta(F)=1$}

Throughout this section, let $F$ be a disconnected graph with $\delta(F) = 1$, and let $n = |F|$, which implies $n \ge 4$.

\subsection{The Graph $X_{2l}$}

For any fixed integer $l > n$, we define the graph $X_{2l}$ as follows:
\[
\begin{aligned}
	V(X_{2l}) &= \{1, 2, \dots, 2l\}, \\[1mm]
	E(X_{2l}) &= \bigl\{ (i,j) \bigm\vert 1 \le i < j \le l \bigr\} \cup \bigl\{ (1, 2l) \bigr\} \  \\[1mm]
	&\quad \cup \bigl\{ (i,j) \bigm\vert l+1 \le i \le 2l-1, \, i-l+1 \le j \le l \bigr\}.
\end{aligned}
\]

Structurally, the vertices $1, 2, \dots, 2l-1$ are partitioned into two levels 
(see Fig.~\ref{fig1}): the upper level forms a clique on $1, 2, \dots, l$, whereas the lower one  
comprises an independent set on $l+1, l+2, \dots, 2l-1$. Each vertex in the lower level 
is adjacent to all upper-level vertices located directly above or to the right of it. Finally, 
the vertex $2l$ is a pendant vertex adjacent to $1$.

\begin{figure}[ht]
	\centering
	\includegraphics[width=10cm]{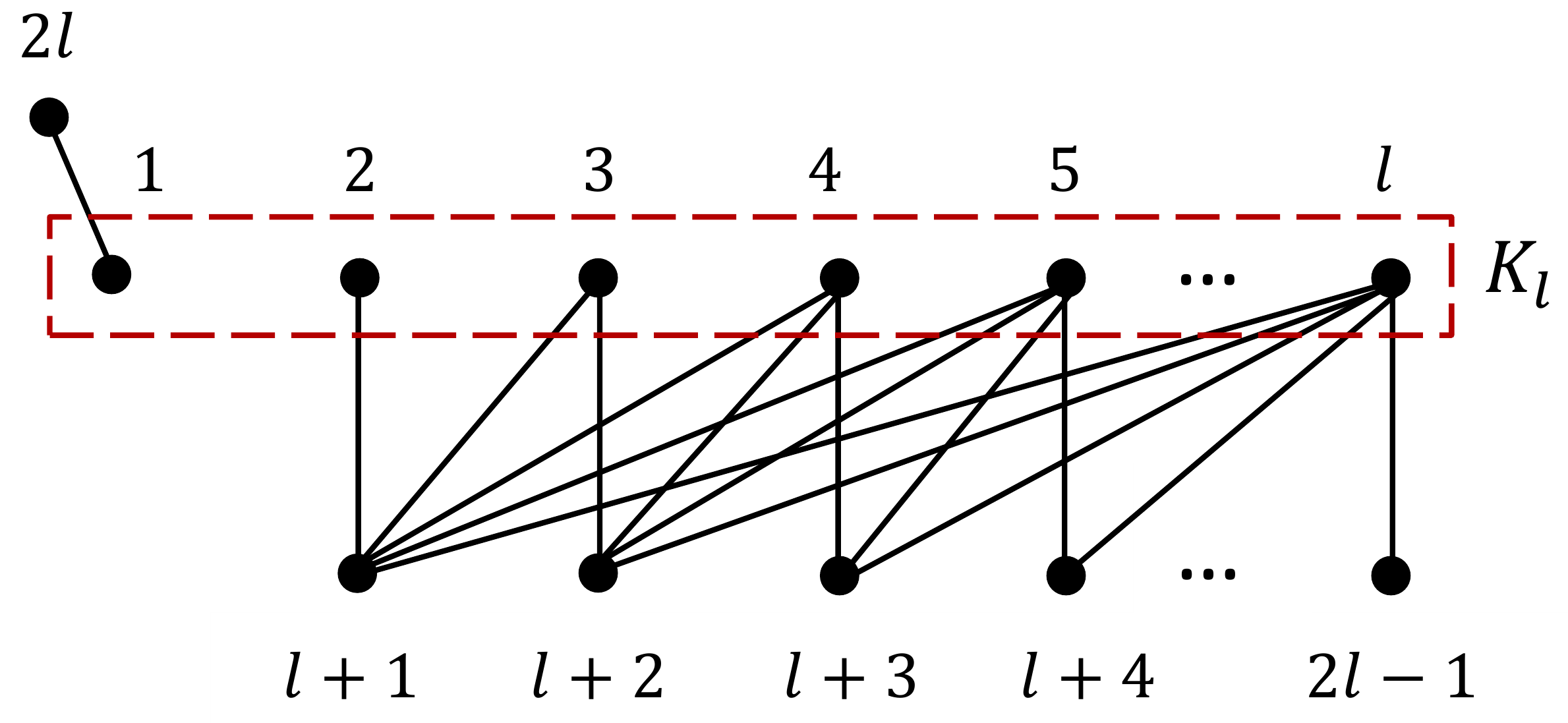}
	\caption{\small The graph $X_{2l}$.}
	\label{fig1}
\end{figure}

For each vertex $i \in V(X_{2l})$, let $x_i$ denote its $F$-degree in $X_{2l}$.
\medskip

In the following lemmas, we establish the relations between the $F$-degrees 
of vertices in $X_{2l}$, sometimes under additional constraints on the parameter~$l$. 
The proofs of Lemmas~\ref{l3.1}, \ref{l3.2}, \ref{l3.5}, and \ref{l3.9} are omitted, 
as they are essentially identical to those of Lemmas 2.1, 2.2, 2.4, and 2.5 from~\cite{6}, respectively.

\begin{lemma}\label{l3.1}
	Let $i \in \{2, 3, \dots, l-1\}$. Then $x_{i+1} > x_i$.
\end{lemma}

\begin{lemma}\label{l3.2}
	Let $i \in \{l+1, l+2, \dots, 2l-2\}$. Then $x_i > x_{i+1}$.
\end{lemma}

\begin{lemma}\label{l3.3}
	The inequalities $x_{l+1} < x_1$ and $x_{l+1} < x_2$ hold.
\end{lemma}
\begin{proof}
	In the graph $X_{2l} \setminus \{(1, 2l)\}$, the vertices $1$ and $l+1$ have the same neighborhood $\{2, 3, \dots, l\}$ and are therefore symmetric. This yields
	\[
	x_1 - x_{l+1} = |\mathcal{A}_1|, \quad \text{where } \mathcal{A}_1 = \{ G \subseteq X_{2l} \mid G \cong F, \, (1, 2l) \in E(G), \, l+1 \notin V(G) \}.
	\]
	Since $l > n$, the clique on vertices $1, 2, \dots, n-1$ excludes $l+1$. Thus, there exists a subgraph in $\mathcal{A}_1$ with the vertex set $\{1, 2, \dots, n-1\} \cup \{2l\}$, showing that $\mathcal{A}_1 \neq \varnothing$ and $x_{l+1} < x_1$.
	
	Next, in the graph $X_{2l} \setminus \{(1, 2)\}$, the vertices $2$ and $l+1$ have the same closed neighborhood $\{2, 3, \dots, l+1\}$ and are therefore symmetric. This yields
	\[
	x_2 - x_{l+1} = |\mathcal{A}_2|, \quad \text{where } \mathcal{A}_2 = \{ G \subseteq X_{2l} \mid G \cong F, \, (1, 2) \in E(G), \, l+1 \notin V(G) \}.
	\]
	Since $l > n$, the clique on $n$ vertices $1, 2, \dots, n$ excludes $l+1$ 
	and contains a subgraph from $\mathcal{A}_2$. Hence, $\mathcal{A}_2 \neq \varnothing$, which implies $x_{l+1} < x_2$.
\end{proof}

\begin{lemma}\label{l3.4}
	Suppose that $F$ is a 1-regular graph. Then $x_1 > x_2$.
\end{lemma}

\begin{proof}
	In the graph $X_{2l} \setminus \{(1, 2l), (2, l+1)\}$, the vertices $1$ and $2$ have the same closed neighborhood $\{1, 2, \dots, l\}$ and are therefore symmetric. This yields
	\[
	\begin{aligned}
		x_1 - x_2 &= |\mathcal{K}| - |\mathcal{M}|, \quad \text{where} \\
		\mathcal{K} &= \{ G \subseteq X_{2l} \mid G \cong F, \, (1, 2l) \in E(G), \, 2 \notin V(G) \}, \\
		\mathcal{M} &= \{ H \subseteq X_{2l} \mid H \cong F, \, (2, l+1) \in E(H), \, 1 \notin V(H) \}.
	\end{aligned}
	\]
	To prove $x_1 > x_2$, it suffices to show that $|\mathcal{K}| > |\mathcal{M}|$.
	
First, any subgraph $H \in \mathcal{M}$ excludes the vertex $2l$; otherwise, $H$ 
would also contain its neighbor, vertex $1$, contradicting the condition $1 \notin V(H)$. 
Furthermore, since $H \cong F$  and $F$ is a \mbox{1-regular}  graph, every connected component of $F$ is a $K_2$-component, meaning the vertices $2$ and $l+1$ are adjacent only 
to each other in $H$.
We define an injection from $\mathcal{M}$ to $\mathcal{K}$ that maps each subgraph 
$H \in \mathcal{M}$ to a subgraph in $\mathcal{K}$ by the following rule:
	\[
	\begin{aligned}
		2 &\to 2l, \quad l+1 \to 1, \quad i \to i \quad \text{for } i \in V(H) \setminus \{2, l+1\}, \\
		(2, l+1) &\to (2l, 1), \quad (i, j) \to (i, j) \quad \text{for } (i, j) \in E(H) \setminus \{(2, l+1)\}.
	\end{aligned}
	\]
	This mapping is not surjective because the image of any subgraph from $\mathcal{M}$ excludes the vertex $l+1$, whereas $\mathcal{K}$ contains at least one subgraph including $l+1$. Indeed, since $l > n$, such a subgraph can be formed by choosing two of its $K_2$-components on the vertex sets $\{1, 2l\}$ and $\{l, l+1\}$, while the remaining components are chosen from the clique induced by the vertices $3, 4, \dots, n-2$. Thus, $|\mathcal{K}| > |\mathcal{M}|$.
\end{proof}
		
\begin{lemma}\label{l3.5}
	If $\binom{l-2}{n-2} > n! \binom{2l-4}{n-3}$ and $\Delta(F) > 1$, then $x_2 > x_1$.
\end{lemma}

\begin{lemma}\label{l3.6}
	If $\binom{l-2}{n-2} > n! \binom{2l-4}{n-3}$ and $F$ is a 1-regular graph, then $x_3 > x_1$.
\end{lemma}

\begin{proof}
	In the graph $X_{2l} \setminus \{(1, 2l), (3, l+1), (3, l+2)\}$, the vertices $1$ and $3$ have the same closed neighborhood $\{1, 2, \dots, l\}$ and are therefore symmetric. This yields
\[
\begin{aligned}
	x_3 - x_1 &= |\mathcal{L}| - |\mathcal{N}|, \quad \text{where} \\
	\mathcal{L} &= \{ H \subseteq X_{2l} \mid H \cong F, \, \{(3, l+1), (3, l+2)\} \cap E(H) \neq \varnothing, \, 1 \notin V(H) \}, \\
	\mathcal{N} &= \{ G \subseteq X_{2l} \mid G \cong F, \, (1, 2l) \in E(G), \, 3 \notin V(G) \}.
\end{aligned}
\]
To prove $x_3 > x_1$, it suffices to show that $|\mathcal{L}| > |\mathcal{N}|$. To this end, we partition $\mathcal{L}$ and $\mathcal{N}$ into disjoint subsets: 
\[
\begin{aligned}
	\mathcal{L}_1 &= \{ H \in \mathcal{L} \mid (3, l+2) \notin E(H) \}, \quad &\mathcal{L}_2 &= \{ H \in \mathcal{L} \mid (3, l+2) \in E(H) \}, \\
	\mathcal{N}_1 &= \{ G \in \mathcal{N} \mid l+1 \notin V(G) \}, \quad &\mathcal{N}_2 &= \{ G \in \mathcal{N} \mid l+1 \in V(G) \}.
\end{aligned}
\]

Since $G \cong F$ for any subgraph $G \in \mathcal{N}_1$ and $F$ is a 1-regular graph, the vertices $1$ and $2l$ are adjacent only to each other in $G$. 
We define an injection from $\mathcal{N}_1$ to $\mathcal{L}_1$ that maps each subgraph 
$G \in \mathcal{N}_1$ to a subgraph in $\mathcal{L}_1$ given by:
\[
\begin{aligned}
	1 &\to l+1, \quad 2l \to 3, \quad i \to i \quad \text{for } i \in V(G) \setminus \{1, 2l\}, \\
	(1, 2l) &\to (l+1, 3), \quad (i, j) \to (i, j) \quad \text{for } (i, j) \in E(G) \setminus \{(1, 2l)\}.
\end{aligned}
\]
This implies $|\mathcal{L}_1| \ge |\mathcal{N}_1|$.

To upper-bound $|\mathcal{N}_2|$, we observe that every subgraph in $\mathcal{N}_2$ contains the vertices $1, l+1$, and $2l$, while excluding $3$. The remaining $n-3$ vertices can be selected from $V(X_{2l}) \setminus \{1, 3, l+1, 2l\}$ in exactly $\binom{2l-4}{n-3}$ ways. Since there are at most $n!$ subgraphs isomorphic to $F$ on any fixed set of $n$ vertices, we obtain
\[	
|\mathcal{N}_2| \le n! \binom{2l-4}{n-3}.
\]

To lower-bound $|\mathcal{L}_2|$, we note that since $l > n$, any subset~$V$ of $n-2$ upper-level vertices distinct from $1$ and $3$ induces a clique. Thus, there exists a subgraph in $\mathcal{L}_2$ with vertex set $V \cup \{3, l+2\}$ composed of a $K_2$-component on the vertex set $\{3, l+2\}$ and the remaining components chosen from the clique induced by $V$. Since there are $\binom{l-2}{n-2}$ choices for $V$, it follows that
\[
|\mathcal{L}_2| \ge \binom{l-2}{n-2}.
\]

If $\binom{l-2}{n-2} > n! \binom{2l-4}{n-3}$, then $|\mathcal{L}_2| > |\mathcal{N}_2|$. Combined with $|\mathcal{L}_1| \ge |\mathcal{N}_1|$, this immediately yields $|\mathcal{L}| > |\mathcal{N}|$.
\end{proof}

\begin{lemma}\label{l3.7}
	If every pendant vertex of $F$ lies in a $K_2$-component, then $x_{2l} > x_{2l-1}$.
\end{lemma}
\begin{proof}
	We have $x_{2l} - x_{2l-1} = |\mathcal{B}| - |\mathcal{A}|$, where
	\[
	\begin{aligned}
		\mathcal{B} &= \{ G \subseteq X_{2l} \mid G \cong F, \, 2l \in V(G), \, 2l-1 \notin V(G) \}, \\
		\mathcal{A} &= \{ H \subseteq X_{2l} \mid H \cong F, \, 2l-1 \in V(H), \, 2l \notin V(H) \}.
	\end{aligned}
	\]
	To prove $x_{2l} > x_{2l-1}$, it suffices to show that $|\mathcal{B}| > |\mathcal{A}|$.
	
	Let $H \in \mathcal{A}$. In $H$, the vertex $2l-1$ is a pendant vertex adjacent to $l$. Since $H \cong F$ and every pendant vertex of $F$ belongs to a $K_2$-component, it follows that $l$ is adjacent only to $2l-1$ in $H$. 
	We define an injection from $\mathcal{A}$ to $\mathcal{B}$ that maps each subgraph $H \in \mathcal{A}$ to a subgraph in $\mathcal{B}$ given by:
	\[
	\begin{aligned}
		l &\to 1, \quad 2l-1 \to 2l, \quad 1 \to l \quad \text{if } 1 \in V(H), \\
		i &\to i \quad \text{for } i \in V(H) \setminus \{1, l, 2l-1\}, \\
		(l, 2l-1) &\to (1, 2l), \\
		(1, i) &\to (l, i) \quad \text{for } (1, i) \in E(H), \\
		(i, j) &\to (i, j) \quad \text{for } (i, j) \in E(H) \text{ with } i, j \notin \{1, l\}.
	\end{aligned}
	\]
	We show that this mapping is not surjective. First, observe that the image of any subgraph from $\mathcal{A}$ does not contain the edge $(l, l+1)$; otherwise, there exists a subgraph in $\mathcal{A}$ with the edge $(1, l+1)$, which is impossible. 	
	On the other hand, since $l > n$, $\mathcal{B}$ contains a subgraph with the vertex set $\{1, 2, \dots, n-3\} \cup \{l, l+1, 2l\}$ excluding $2l-1$ that, along with a $K_2$-component on the vertex set $\{1, 2l\}$, contains the edge $(l, l+1)$. The existence of such a subgraph follows from the fact that the vertices $2, 3, \dots, n-3, l$, and $l+1$ induce a clique in $X_{2l}$. Thus, $|\mathcal{B}| > |\mathcal{A}|$.
\end{proof}

\begin{lemma}\label{l3.8}
	If $\binom{l-2}{n-2} > n! \binom{2l-4}{n-3}$ and every pendant vertex of $F$ lies in a $K_2$-component, then $x_{2l-2} > x_{2l}$.
\end{lemma}

\begin{proof}
	We have $x_{2l-2} - x_{2l} = |\mathcal{W}| - |\mathcal{Z}|$, where
	\[
	\begin{aligned}
		\mathcal{W} &= \{ H \subseteq X_{2l} \mid H \cong F, \, 2l-2 \in V(H), \, 2l \notin V(H) \}, \\
		\mathcal{Z} &= \{ G \subseteq X_{2l} \mid G \cong F, \, 2l \in V(G), \, 2l-2 \notin V(G) \}.
	\end{aligned}
	\]
	To prove $x_{2l-2} > x_{2l}$, it suffices to show that $|\mathcal{W}| > |\mathcal{Z}|$. 
	For this purpose, we partition $\mathcal{W}$ and $\mathcal{Z}$ into disjoint subsets:
	\[
	\begin{aligned}
		\mathcal{W}_1 &= \{ H \in \mathcal{W} \mid (2l-2, l-1) \notin E(H) \}, \quad &\mathcal{W}_2 &= \{ H \in \mathcal{W} \mid (2l-2, l-1) \in E(H) \}, \\
		\mathcal{Z}_1 &= \{ G \in \mathcal{Z} \mid l \notin V(G) \}, \quad &\mathcal{Z}_2 &= \{ G \in \mathcal{Z} \mid l \in V(G) \}.
	\end{aligned}
	\]
	
	Since $G \cong F$ for any subgraph $G \in \mathcal{Z}_1$ and every pendant vertex of $F$ lies in  
\mbox{a $K_2$-component}, the vertex $1$ is adjacent only to $2l$ in $G$. We define an injection from $\mathcal{Z}_1$ to $\mathcal{W}_1$ that maps each subgraph $G \in \mathcal{Z}_1$ to a subgraph in $\mathcal{W}_1$ given by:
	\[
	\begin{aligned}
		1 &\to 2l-2, \quad 2l \to l, \quad i \to i \quad \text{for } i \in V(G) \setminus \{1, 2l\}, \\
		(1, 2l) &\to (2l-2, l), \quad (i, j) \to (i, j) \quad \text{for } (i, j) \in E(G) \setminus \{(1, 2l)\}.
	\end{aligned}
	\]
	This implies $|\mathcal{W}_1| \ge |\mathcal{Z}_1|$.
	
	To upper-bound $|\mathcal{Z}_2|$, we observe that every subgraph in $\mathcal{Z}_2$ contains the vertices $1, l, 2l$ and excludes $2l-2$. The remaining $n-3$ vertices can be selected directly from \mbox{$V(X_{2l}) \setminus \{1, l, 2l-2, 2l\}$} in exactly $\binom{2l-4}{n-3}$ ways. Since there are at most $n!$ subgraphs isomorphic to $F$ on any fixed set of $n$ vertices, we obtain
	\[		
	|\mathcal{Z}_2| \le n! \binom{2l-4}{n-3}.
	\]
	
	To lower-bound $|\mathcal{W}_2|$, we note that since $l > n$, any subset~$V$ of $n-2$ upper-level vertices distinct from $l-1$ and $l$ induces a clique. Thus, there exists a subgraph in $\mathcal{W}_2$ with vertex set $V \cup \{l-1, 2l-2\}$ composed of a $K_2$-component on the vertex set $\{l-1, 2l-2\}$ and the remaining components chosen from the clique induced by $V$. Since there are $\binom{l-2}{n-2}$ choices for $V$, it follows that
	\[
	|\mathcal{W}_2| \ge \binom{l-2}{n-2}.
	\]
	
	If $\binom{l-2}{n-2} > n! \binom{2l-4}{n-3}$, then  $|\mathcal{W}_2| > |\mathcal{Z}_2|$. Combined with $|\mathcal{W}_1| \ge |\mathcal{Z}_1|$, this  immediately yields $|\mathcal{W}| > |\mathcal{Z}|$.
\end{proof}

\begin{lemma}\label{l3.9}
	If $2 \mid l$, $\frac{l}{2} \binom{l/2-1}{n-3} > n! \binom{2l-4}{n-3}$, and $F$ contains a connected component with a pendant vertex other than $K_2$, then $x_{2l-1} > x_{2l}$.
\end{lemma}

\subsection{Main result for $\delta(F)=1$}
We require the following two lemmas, 
which were established as Lemmas~2.6 and~2.7 in~\cite{6}, respectively. 

\begin{lemma}\label{l3.10}
	Let $k_1, k_2, m_1, m_2, x \in \mathbb{N}$, $b_1, b_2 \in \mathbb{Z}$, $y \in \mathbb{Z}^+$, $x > y$, $r_1 > 0$, and $r_2 \in \mathbb{R}$. 
	Then there exists a value $N_1$ such that for every positive integer  $l$ with  $l \ge N_1$ and   $\operatorname{LCM}(m_1, m_2) \mid l$, the following inequality holds:
	\[
	r_1 \binom{\frac{k_1}{m_1} l + b_1}{x} > r_2 \binom{\frac{k_2}{m_2} l + b_2}{y}.
	\]
\end{lemma}

\begin{lemma}\label{l3.11}
	Let $k_1, k_2, m_1, m_2 \in \mathbb{N}$, $b_1, b_2 \in \mathbb{Z}$, $x \in \mathbb{Z}^+$, $p_1 > 0$, and $r_1, r_2 \in \mathbb{R}$. 
	Then there exists a value $N_2$ such that for every positive integer $l$ with $l \ge N_2$ and  $\operatorname{LCM}(m_1, m_2) \mid l$, the following inequality holds:
	\[
	(p_1 l + r_1) \binom{\frac{k_1}{m_1} l + b_1}{x} > r_2 \binom{\frac{k_2}{m_2} l + b_2}{x}.
	\]
\end{lemma}

\begin{theorem} \label{t3.1}
	For any disconnected graph $F$ with $\delta(F) = 1$, there exist infinitely many $F$-irregular graphs.
\end{theorem}

\begin{proof}
	Let $n = |F|$. Consider the following Condition 1:
	\[
	2 \mid l, \quad l > n, \quad \binom{l-2}{n-2} > n! \binom{2l-4}{n-3}, \quad \frac{l}{2} \binom{l/2-1}{n-3} > n! \binom{2l-4}{n-3}.
	\]
	We show that Condition 1 is satisfied by infinitely many values of $l$.
	
	Indeed, by Lemmas~\ref{l3.10} and \ref{l3.11}, there exist numbers $N_1$ and $N_2$ such that the inequality $\binom{l-2}{n-2} > n! \binom{2l-4}{n-3}$ holds for all positive integers $l \ge N_1$, and the inequality $\frac{l}{2} \binom{l/2-1}{n-3} > n! \binom{2l-4}{n-3}$ holds for all even positive integers $l \ge N_2$. Therefore, Condition 1 holds for any even integer $l \ge \max(N_1, N_2, n+1)$. Clearly, there are infinitely many such values of $l$.
	
	Now, consider the graph $X_{2l}$ for an arbitrary $l$ satisfying Condition 1. 
	
	If $F$ is a 1-regular graph, then Lemmas \ref{l3.1}--\ref{l3.4} and \ref{l3.6}--\ref{l3.8} imply
	\vspace{-1.2mm}
	\[
	x_l > x_{l-1} > \dots > x_3 > x_1 > x_2 > x_{l+1} > x_{l+2} > \dots > x_{2l-2} > x_{2l} > x_{2l-1}.\vspace{-1.2mm}
	\]
	
	If $\Delta(F) > 1$ and every pendant vertex of $F$ lies in a $K_2$-component, then according to \mbox{Lemmas~\ref{l3.1}--\ref{l3.3}}, \ref{l3.5}, \ref{l3.7}, and \ref{l3.8}, we have
	\vspace{-1.2mm}
	\[
	x_l > x_{l-1} > \dots > x_3 > x_2 > x_1 > x_{l+1} > x_{l+2} > \dots > x_{2l-2} > x_{2l} > x_{2l-1}.
	\vspace{-1.2mm}
	\]
	
	If $F$ contains a connected component with a pendant vertex other than $K_2$, then by \mbox{Lemmas~\ref{l3.1}--\ref{l3.3}}, \ref{l3.5}, and \ref{l3.9}, we obtain 
	\vspace{-1.2mm}
	\[
	x_l > x_{l-1} > \dots > x_3 > x_2 > x_1 > x_{l+1} > x_{l+2} > \dots > x_{2l-2} > x_{2l-1} > x_{2l}. \vspace{-1.2mm}
	\]
	
	Thus, in all cases, the graph $X_{2l}$ of order $2l$ is $F$-irregular. Since Condition 1 is satisfied by infinitely many values of $l$, this establishes the existence of an infinite family of pairwise non-isomorphic $F$-irregular graphs.
\end{proof}

\section{The case of disconnected graphs $F$ with $\delta(F) \ge 2$}

Throughout this section, let $F$ be a disconnected graph with $\delta(F) = t \ge 2$, and let $n = |F|$, which implies $n \ge 6$.

\subsection{The Graph $G_{2l-1}$}

Let $l$ be a positive integer such that $2n^2 \mid l$. We construct the graph $G_{2l-1}$ as follows:
\[
\begin{aligned}
	V(G_{2l-1}) &= \{1, 2, \dots, 2l-1\}, \quad 
	E(G_{2l-1}) &= E_0 \cup E_R,
\end{aligned}
\]
where $E_0 = \bigl\{ (i, j) \bigm\vert 1 \le i < j \le l \bigr\} \cup \bigl\{ (i, j) \bigm\vert l+1 \le i \le 2l-1, \, i-l+1 \le j \le l \bigr\}$ is the baseline edge set, and $E_R$ is the set of additional edges associated with the red vertices, specified below.

Consider a sequence of $\frac{t(t-1)}{2}$ vertices in $G_{2l-1}$ labeled as $1 + \frac{l}{n^2} \cdot j$, where $j \in \mathbb{Z}^+$ and $j \le \frac{t(t-1)}{2} - 1$. We designate these vertices as \emph{red}, and all other vertices in $G_{2l-1}$ as \emph{black}. By Lemma~\ref{l4.1}, the labels of all red vertices do not exceed $\frac{l}{2}$.
\smallskip

Formally, the additional edge set $E_R$ is constructed by connecting the first $t-1$ red vertices to vertex $2l-1$, the next $t-2$ red vertices to vertex $2l-2$, the subsequent $t-3$ red vertices to vertex $2l-3$, and so on. Finally, the red vertex labeled as $1 + \frac{l}{n^2} \cdot \left(\frac{t(t-1)}{2}-1\right)$ is connected to vertex $2l-t+1$.
\smallskip

 Note that $E_0 \cap E_R = \varnothing$. Indeed, under $E_0$, the vertices $2l-t+1, 2l-t+2, \dots, 2l-1$ are adjacent only to vertices from the set $\{l-t+2, l-t+3, \dots, l\}$. Since $l-t+2 > \frac{l}{2}$ holds for any $l$ satisfying $2n^2 \mid l$ (given $n > t \ge 2$), the set $\{l-t+2, l-t+3, \dots, l\}$ contains no red vertices.

Thus, the vertices $2l-t, 2l-t+1, \dots, 2l-1$ in $G_{2l-1}$ have degree $t$, while all other vertices have degrees strictly greater than $t$.

To visualize $G_{2l-1}$, we arrange its vertices in two levels (see Fig.~\ref{fig2}), where the upper level contains \mbox{$1, 2, \dots, l$} and the lower level contains \mbox{$l+1, l+2, \dots, 2l-1$}. In this representation, the baseline edge set $E_0$ consists of all edges between upper-level vertices, alongside edges connecting each lower-level vertex $i$ to the vertex directly above it and to all upper-level vertices located strictly to the right of $i$.

\begin{figure}[ht]
	\centering
	\includegraphics[width=10cm]{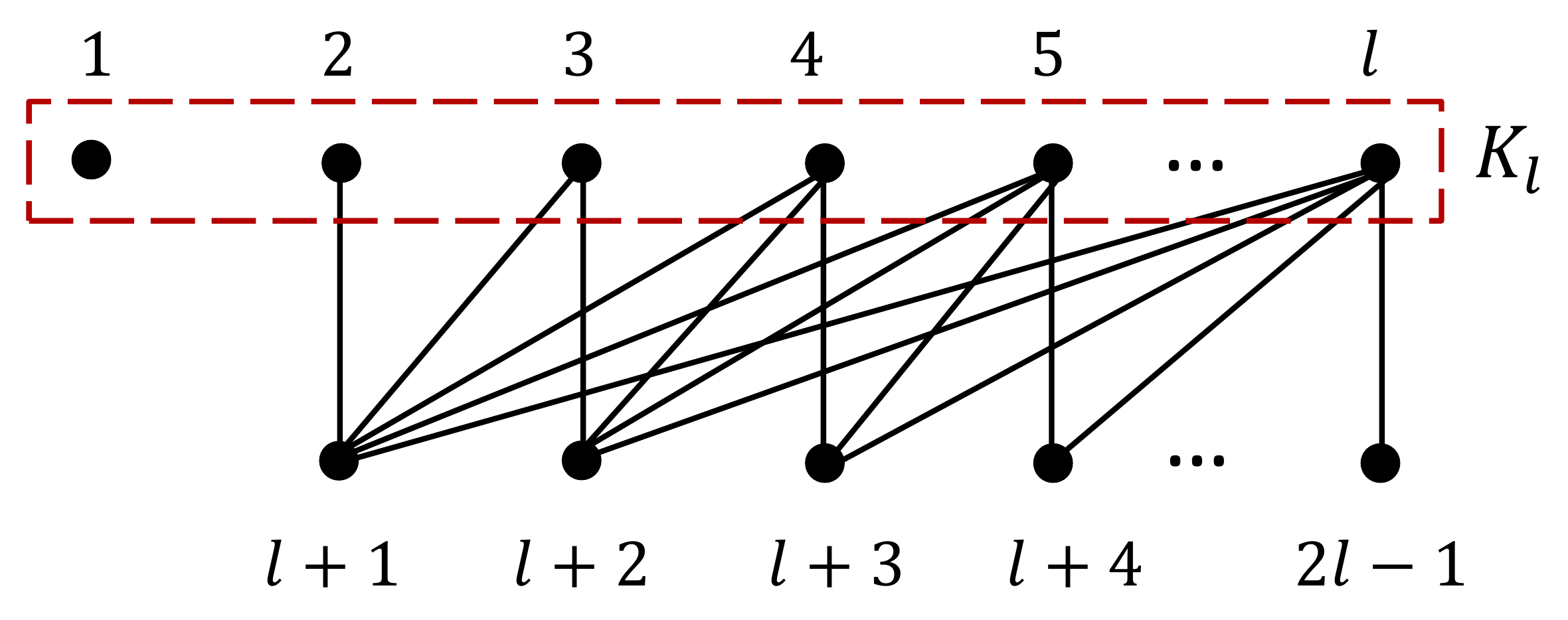}
	\caption{\small The baseline edge set $E_0$ of the graph $G_{2l-1}$.}
	\label{fig2}
\end{figure}

The additional edge set $E_R$ connecting the vertices \mbox{$2l-t+1, 2l-t+2, \dots, 2l-1$} to the upper-level red vertices is illustrated in Fig.~\ref{fig3}.

\begin{figure}[ht]
	\centering
	\includegraphics[width=14cm]{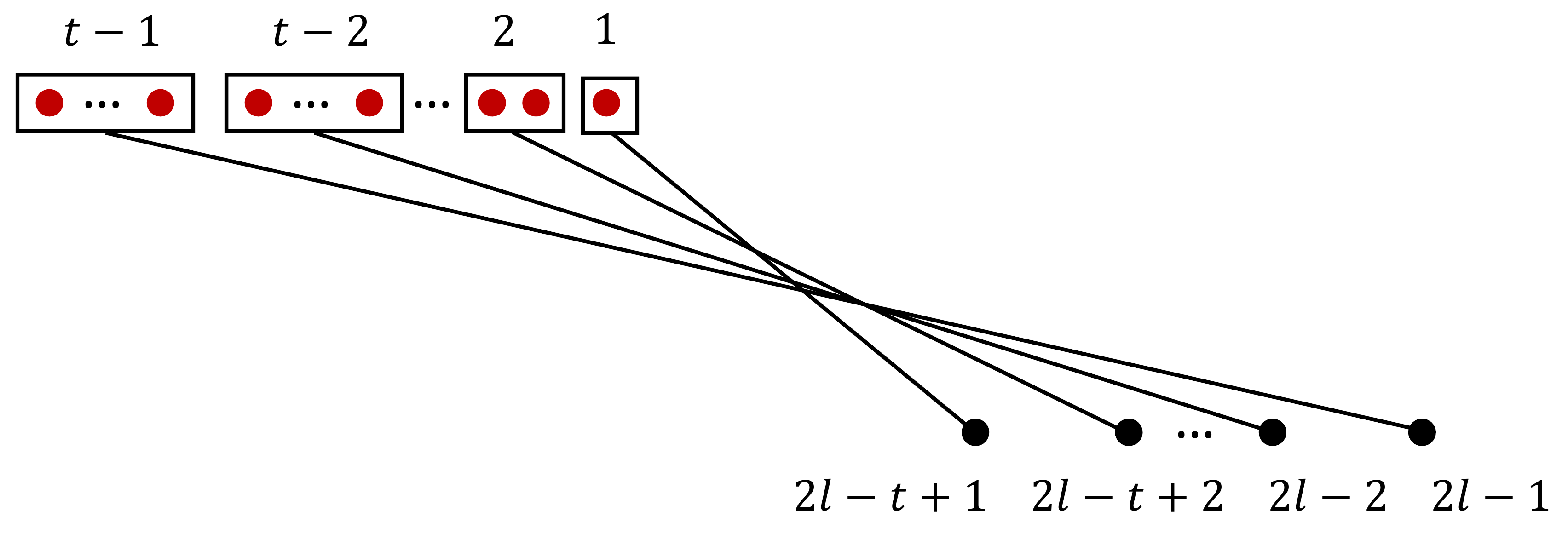}
	\caption{\small The additional edge set $E_R$ of the graph $G_{2l-1}$ associated with the red vertices.}
	\label{fig3}
\end{figure}

Next, for each vertex $i \in V(G_{2l-1})$, let $g_i$ denote its $F$-degree in $G_{2l-1}$.
\smallskip

In the following lemmas, we establish the order of vertex $F$-degrees in $G_{2l-1}$. 
The proofs of Lemmas~\ref{l4.1}--\ref{l4.4} and~\ref{l4.6} are omitted, as their 
arguments carry over with minimal changes from Lemmas~3.2--3.6 of~\cite{6}, respectively.

\begin{lemma}\label{l4.1}
	In the graph $G_{2l-1}$, the label of every red vertex does not exceed $\frac{l}{2}$.
\end{lemma}

\begin{lemma}\label{l4.2}
	If $\binom{l/2-1}{n-2} > n! \binom{2l-t-3}{n-t-1}$, then $g_{i+1} > g_i$ for any $i \in \{1, 2, \dots, l-1\}$.
\end{lemma}

\begin{lemma}\label{l4.3}
	The inequality $g_1 > g_{l+1}$ holds.
\end{lemma}

\begin{lemma}\label{l4.4}
	For any $i \in \{l+1, l+2, \dots, 2l-t-1\}$, we have $g_i > g_{i+1}$.
\end{lemma}

\begin{lemma}\label{l4.5}
	If $\binom{l-t}{n-t-1} > n! t \binom{2l-t-4}{n-t-2}$, then $g_{2l-t-1} > g_{2l-1}$.
\end{lemma}

\begin{proof}
	In $G_{2l-1}$, consider the vertices $2l-t-1$ and $2l-1$. By construction, vertex $2l-t-1$ is adjacent to exactly $t+1$ vertices, namely $l-t, l-t+1, \dots, l$, while vertex $2l-1$ is adjacent to exactly $t$ vertices
	\vspace{-3mm}
	\[
	1, \, 1+\frac{l}{n^2}, \, 1+\frac{2l}{n^2}, \dots, 1+\frac{(t-2)l}{n^2}, \, l.
	\] 
	We have $g_{2l-t-1} - g_{2l-1} = |\mathcal{P}| - |\mathcal{Q}|$, where
	\[
	\begin{aligned}
		\mathcal{P} &= \{ H \subseteq G_{2l-1} \mid H \cong F, \, 2l-t-1 \in V(H), \, 2l-1 \notin V(H) \}, \\
		\mathcal{Q} &= \{ G \subseteq G_{2l-1} \mid G \cong F, \, 2l-1 \in V(G), \, 2l-t-1 \notin V(G) \}.
	\end{aligned}
	\]
	To prove $g_{2l-t-1} > g_{2l-1}$, it suffices to show that $|\mathcal{P}| > |\mathcal{Q}|$.
	\smallskip
	
	Since vertex $2l-1$ has degree $t$ in $G_{2l-1}$, its neighborhood is preserved in every subgraph from $\mathcal{Q}$. Let $V_t = \{l-t, l-t+1, \dots, l-1\}$. We partition $\mathcal{P}$ and $\mathcal{Q}$ into disjoint subsets:
	\[
	\begin{aligned}
		\mathcal{P}_1 &= \{ H \in \mathcal{P} \mid l-t \notin V(H) \}, \quad &\mathcal{P}_2 &= \{ H \in \mathcal{P} \mid l-t \in V(H) \}, \\
		\mathcal{Q}_1 &= \{ G \in \mathcal{Q} \mid V(G) \cap V_t = \varnothing \}, \quad &\mathcal{Q}_2 &= \{ G \in \mathcal{Q} \mid V(G) \cap V_t \neq \varnothing \}.
	\end{aligned}
	\]
	
First, we establish $|\mathcal{P}_1| \ge |\mathcal{Q}_1|$. Let 
	\vspace{-2mm}
	\[
	L = \Bigl\{1, 1+\frac{l}{n^2}, 1+\frac{2l}{n^2}, \dots, 1+\frac{(t-2)l}{n^2}, 2l-1\Bigr\}, \quad J = \{0, 1, \dots, t-2\}.
	\]
We construct an injection from $\mathcal{Q}_1$ to $\mathcal{P}_1$ by mapping each subgraph $G \in \mathcal{Q}_1$ via the vertex and edge assignments:
\vspace{-2mm}
\[
\begin{aligned}
	2l-1 \to 2l-t-1, \quad &1+\frac{lj}{n^2} \to l-t+j+1, \quad s \to s, \\[1mm]
	(2l-1, l) &\to (2l-t-1, l), \\
	\left(2l-1, 1+\frac{lj}{n^2}\right) &\to \left(2l-t-1, l-t+j+1\right), \\
	\left(1+\frac{lj}{n^2}, 1+\frac{lk}{n^2}\right) \in E(G) &\to (l-t+j+1, l-t+k+1), \\
	\left(s, 1+\frac{lj}{n^2}\right) \in E(G) &\to (s, l-t+j+1), \\
	(s, z) \in E(G) &\to (s, z),
\end{aligned}
\]
where $j, k \in J$ and $s, z \in V(G) \setminus L$.
This mapping yields $|\mathcal{P}_1| \ge |\mathcal{Q}_1|$.
\smallskip
	
Next, we prove $|\mathcal{P}_2| > |\mathcal{Q}_2|$. To upper-bound $|\mathcal{Q}_2|$, we note that every subgraph in $\mathcal{Q}_2$ contains vertex $2l-1$, its $t$ fixed upper-level neighbors, and excludes vertex $2l-t-1$. Furthermore, each such subgraph includes at least one vertex from $V_t$, for which there are exactly $t$ choices, yielding \mbox{$\binom{2l-t-4}{n-t-2}$} choices for the remaining \mbox{$n-t-2$} vertices. Since on any fixed set of $n$ vertices there are at most $n!$ subgraphs isomorphic to $F$, we obtain
\[	
|\mathcal{Q}_2| \le n! t \binom{2l-t-4}{n-t-2}.
\]
	
To lower-bound $|\mathcal{P}_2|$, we note that since $l \ge 2n^2 > n$, any subset $U$ of $n-t-1$ upper-level vertices disjoint from $V_t \cup \{l\}$ induces a clique. Thus, there exists a subgraph in $\mathcal{P}_2$ with vertex set $U \cup V_t \cup \{2l-t-1\}$, where vertex $2l-t-1$ has degree $t$ in the subgraph. Since there are $\binom{l-t}{n-t-1}$ choices for $U$, it follows that 
\[|\mathcal{P}_2| \ge \binom{l-t}{n-t-1}.\]
	
By assumption, $\binom{l-t}{n-t-1} > n! t \binom{2l-t-4}{n-t-2}$, so $|\mathcal{P}_2| > |\mathcal{Q}_2|$. Combined with $|\mathcal{P}_1| \ge |\mathcal{Q}_1|$, this gives $|\mathcal{P}| > |\mathcal{Q}|$.
\end{proof}

\begin{lemma}\label{l4.6}
	If $\frac{l}{n^2} \binom{l/n^2}{n-t-2} > n! (t-1) \binom{2l-t-4}{n-t-2}$ and $F$ contains a connected component with a vertex of minimum degree $t$ that is not isomorphic to $K_{t+1}$, then $g_i > g_{i+1}$ for all $i \in \{2l-t, 2l-t+1, \dots, 2l-2\}$.
\end{lemma}

\begin{lemma}\label{l4.7}
	If every vertex of minimum degree $t$ in $F$ belongs to a $K_{t+1}$-component, then $g_i < g_{i+1}$ for all $i \in \{2l-t, 2l-t+1, \dots, 2l-2\}$.
\end{lemma}

\begin{proof}
	Fix $i \in \{2l-t, 2l-t+1, \dots, 2l-2\}$ and let $d = t + i - 2l$. In $G_{2l-1}$, vertex $i$ has degree $t$ and is adjacent to $2l-i$ black vertices $i-l+1, i-l+2, \dots, l$ and $d$ red vertices; if $i \ge 2l-t+1$, these red neighbors, in increasing order of their labels, are denoted by $m_1, m_2, \dots, m_d$. Vertex $i+1$, in turn, is adjacent to $2l-i-1$ black vertices $i-l+2, i-l+3, \dots, l$ and $d+1$ red vertices, denoted in increasing order by $k_1, k_2, \dots, k_{d+1}$. For $i \ge 2l-t+1$, by the definition of $E_R$, the constraints on $l$, and Lemma~\ref{l4.1}, these labels satisfy (see Fig.~\ref{fig4})
	\[
	k_1 < k_2 < \dots < k_{d+1} < m_1 < m_2 < \dots < m_d < i-l+1.
	\]
	
	For $i = 2l-t$, vertex $i$ has no red neighbors ($d = 0$  implies $\{m_1, m_2, \dots, m_d\} = \varnothing$), and the sole red neighbor $k_1$ of vertex $i+1$ satisfies $k_1 < i-l+1$.
	
	\begin{figure}[ht]
		\centering
		\includegraphics[width=12cm]{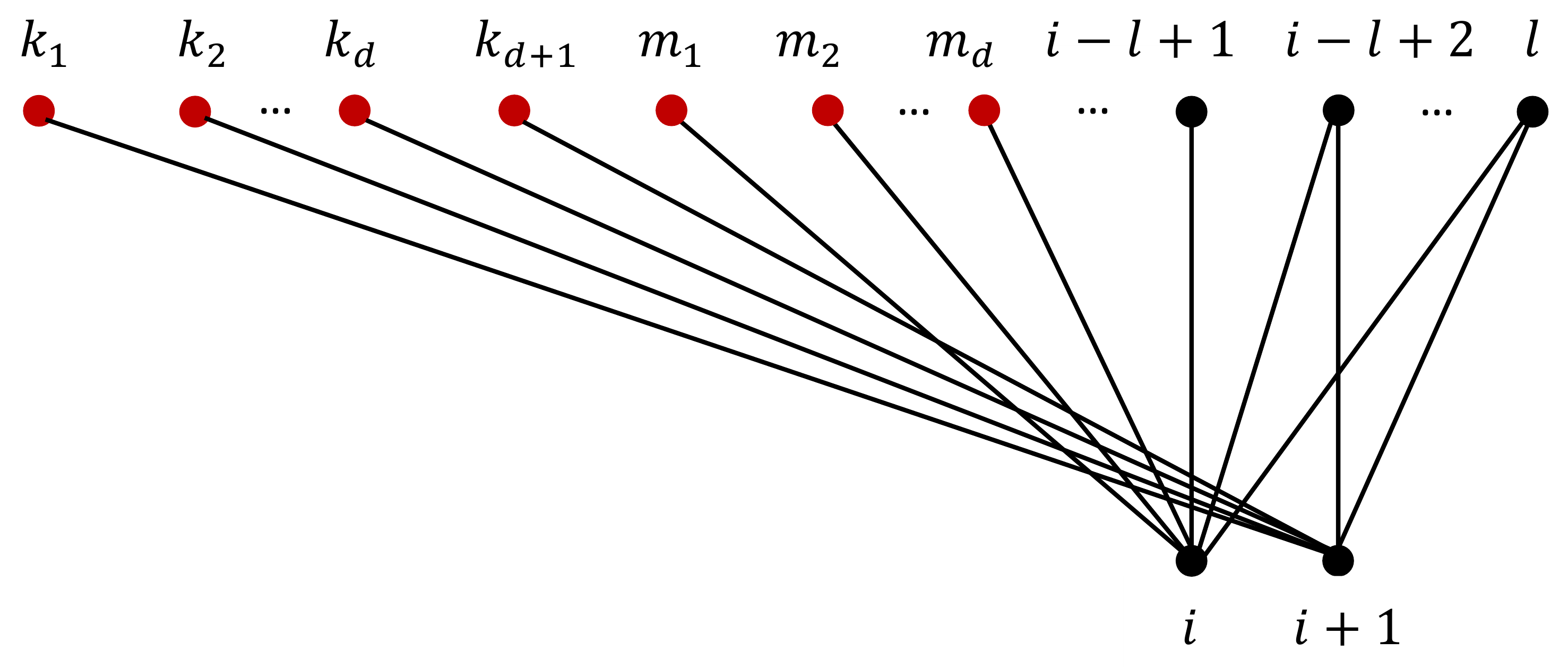}
		\caption{Neighborhoods of $i$ and $i+1$ in $G_{2l-1}$ for $i \in \{2l-t+1, 2l-t+2, \dots, 2l-2\}$.}
		\label{fig4}
	\end{figure}
	
	Next, we have $g_{i+1} - g_i = |T| - |S|$, where
	\[
	\begin{aligned}
		T &= \{ G \subseteq G_{2l-1} \mid G \cong F, \, i+1 \in V(G), \, i \notin V(G) \}, \\
		S &= \{ H \subseteq G_{2l-1} \mid H \cong F, \, i \in V(H), \, i+1 \notin V(H) \}.
	\end{aligned} 
	\]
	To prove $g_i < g_{i+1}$, it suffices to show that $|T| > |S|$. To this end, we define an injection from $S$ to $T$.	
	Since $H \cong F$ and $G \cong F$ for any subgraphs $H \in S$ and $G \in T$, and every vertex of minimum degree in $F$ lies in a $K_{t+1}$-component, these subgraphs must contain $K_{t+1}$-components $C_1 \subseteq H$ and $C_2 \subseteq G$ on the vertex sets:
	\[
	\begin{aligned}
		V(C_1) &= \{i, m_1, m_2, \dots, m_d, i-l+1, i-l+2, \dots, l\}, \\
		V(C_2) &= \{i+1, k_1, k_2, \dots, k_{d+1}, i-l+2, i-l+3, \dots, l\}.
	\end{aligned}
	\]
	Let $I = \{1, 2, \dots, d\}$, where $I = \varnothing$ if $i = 2l-t$, and define the vertex set
	\[
	W = \{i, k_{d+1}\} \cup \{m_p, k_p \mid p \in I\} \cup \{i-l+1, i-l+2, \dots, l\}.
	\]
Each $H \in S$ is mapped to $T$ by replacing $C_1$ with $C_2$, while the rest of $H$ is mapped via:
\[
\begin{aligned}
	k_p \in V(H) \to m_p, \quad k_{d+1} \in V(H) \to i-l+1, &\quad s \to s, \\
	(k_p, k_q) \in E(H) \to (m_p, m_q), \quad (k_p, k_{d+1}) \in E(H) \to &(m_p, i-l+1), \\
	(s, k_p) \in E(H) \to (s, m_p), \quad (s, k_{d+1}) \in E(H) \to (s, i-l+1), &\quad (s, z) \in E(H) \to (s, z),
\end{aligned}
\]
where $p, q \in I$ and $s, z \in V(H) \setminus W$.
	
	To show that the mapping is not surjective, we note that by the definition of $G_{2l-1}$ and the strict inequality $k_{d+1} < i-l+1$, no subgraph $H \in S$ contains the edge $(k_{d+1}, l+k_{d+1})$. Consequently, no image under this mapping can include the edge $(i-l+1, l+k_{d+1})$.
	
Conversely, there exists a subgraph $G^* \in T$ containing the edge $(i-l+1, l+k_{d+1})$. Defining the vertex set
\[
U = \{l+k_{d+1}\} \cup \{k_{d+1}+1, k_{d+1}+2, \dots, i-l+1\},
\]
the constraints $i \ge 2l-t$, $k_{d+1} \le \frac{l}{2}$, and $l \ge 2n^2$ yield
\[	
|U| = i - l - k_{d+1} + 2 \ge 2l - t - l - \frac{l}{2} + 2 > \frac{l}{2} - t > n - t - 1.
\]
In $G^*$, we fix a $K_{t+1}$-component on the vertex set $\{i+1, k_1, \dots, k_{d+1}, i-l+2, \dots, l\}$. Since the vertex set $U$ induces a clique and $|U| > n - t - 1$, it is sufficiently large to accommodate the remaining components of $G^*$ such that the edge $(i-l+1, l+k_{d+1})$ is explicitly included. Consequently, $G^*$ has no preimage in $S$, which implies $|T| > |S|$.
\end{proof}

\subsection{Main result for $\delta(F) \ge 2$}

\begin{theorem} \label{t4.1}
	For any disconnected graph $F$ with $\delta(F) \ge 2$, there exist infinitely many $F$-irregular graphs.
\end{theorem}

\begin{proof}
	Let $n = |F|$ and $t = \delta(F) \ge 2$. Since $F$ is disconnected, we have $n - t - 2 > 0$. We proceed by analyzing two cases based on the structure of $F$.
	
	\textbf{Case 1.} $F$ contains a connected component with a vertex of minimum degree $t$ that is not isomorphic to $K_{t+1}$.
	
	Let Condition 2 be defined by the following system:
	\[
	l \in \mathbb{N}, \quad 2n^2 \mid l, \quad \binom{l/2-1}{n-2} > n! \binom{2l-t-3}{n-t-1}, \quad \dfrac{l}{n^2} \binom{l/n^2}{n-t-2} > n!(t-1)\binom{2l-t-4}{n-t-2}.
	\]
	
	Since $n - 2 > n - t - 1$, Lemma~\ref{l3.10} guarantees the existence of a constant $N_1$ such that the first inequality holds for all even $l \ge N_1$. Similarly, by Lemma~\ref{l3.11}, there exists a constant $N_2$ such that the second inequality holds for all $l \ge N_2$ with $n^2 \mid l$. Therefore, Condition 2 is satisfied by any positive integer $l \ge \max(N_1, N_2)$ such that $2n^2 \mid l$, yielding infinitely many valid values of $l$.
	
	Now, consider the graph $G_{2l-1}$ constructed for an arbitrary $l$ satisfying Condition 2. Lemmas~\ref{l4.2}--\ref{l4.4} and \ref{l4.6} imply:
	\[
	g_l > g_{l-1} > \dots > g_2 > g_1 > g_{l+1} > g_{l+2} > \dots > g_{2l-1}.
	\]
	Thus, the graph $G_{2l-1}$ is $F$-irregular. Because Condition 2 holds for infinitely many positive integers $l$, the family of graphs $G_{2l-1}$ provides infinitely many $F$-irregular graphs.
	
	\textbf{Case 2.} Every vertex of minimum degree $t$ in $F$ lies in a connected component isomorphic to $K_{t+1}$.
	
	Let Condition 3 be defined by the following system:
	\[
	l \in \mathbb{N}, \quad 2n^2 \mid l, \quad \binom{l/2-1}{n-2} > n! \binom{2l-t-3}{n-t-1}, \quad \binom{l-t}{n-t-1} > n! t \binom{2l-t-4}{n-t-2}.
	\]
	
	Since $n - 2 > n - t - 1$ and $n - t - 1 > n - t - 2$, Lemma~\ref{l3.10} guarantees the existence of a constant $N_3$ such that the first inequality holds for all even $l \ge N_3$.\\ Similarly, by Lemma~\ref{l3.11}, there exists a constant $N_4$ such that the second inequality holds for all $l \ge N_4$. Therefore, Condition 3 is satisfied by any positive integer $l \ge \max(N_3, N_4)$ such that $2n^2 \mid l$, yielding infinitely many valid values of $l$.
	
	Now, consider the graph $G_{2l-1}$ constructed for an arbitrary $l$ satisfying Condition 3. Lemmas~\ref{l4.2}--\ref{l4.5} and \ref{l4.7} imply:
	\[
	g_l > g_{l-1} > \dots > g_1 > g_{l+1} > g_{l+2} > \dots > g_{2l-t-1} > g_{2l-1} > g_{2l-2} > \dots > g_{2l-t}.
	\]
	Thus, the graph $G_{2l-1}$ is $F$-irregular. Because Condition 3 holds for infinitely many positive integers $l$, the family of graphs $G_{2l-1}$ provides infinitely many $F$-irregular graphs.
\end{proof}

\section{Proof of Theorem 1.1}

We are now ready to assemble our results to establish the main classification theorem.

\begin{proof}[Proof of Theorem 1.1]
	Let $F$ be a disconnected graph. 
	\smallskip
	
	If $\delta(F) \ge 2$, the result follows immediately from Theorem~\ref{t4.1}. 
	\smallskip
	
	If $\delta(F) = 1$, the result follows directly from Theorem~\ref{t3.1}.
	\smallskip
	
	Now, suppose that $\delta(F) = 0$. 
	\smallskip
	
	If $|E(F)| \le 1$, the non-existence of non-trivial $F$-irregular graphs is established by Lemmas~\ref{l2.1} and~\ref{l2.2}. 
	\smallskip
	
	Thus, let $|E(F)| \ge 2$ and express the graph as $F =H+\overline{K_m} $, where $m$ is a positive integer and $\delta(H) \ge 1$. Since $|E(F)| \ge 2$, we have $|E(H)| \ge 2$. 
	\smallskip
	
	We proceed based on the connectivity of $H$:
	\begin{itemize}
		\item If $H$ is connected, the existence of infinitely many $F$-irregular graphs follows from Theorem~\ref{t2.1} combined with the Simple Theorem~\cite[p.~11]{6} for connected graphs.
		\item If $H$ is disconnected, then either $\delta(H) = 1$ or $\delta(H) \ge 2$. Consequently, the result follows from Theorem~\ref{t2.1} in conjunction with Theorems~\ref{t3.1} and~\ref{t4.1}, respectively.
	\end{itemize}
	
	This completes the proof of Theorem~\ref{t1.1}.
\end{proof}

\end{document}